\documentclass[12pt]{amsart}

\calclayout
\usepackage{graphicx}
\usepackage{amssymb, amsfonts, amsmath, MnSymbol}
\usepackage{amsthm}
\usepackage{epstopdf}
\usepackage[nomarkers,notablist]{endfloat}
\usepackage{varioref}
\usepackage{multicol}
\usepackage{placeins}
\usepackage{enumitem}

\newtheorem{thrm}{Theorem}[section]
\newtheorem{prop}[thrm]{Proposition}
\newtheorem{lem}[thrm]{Lemma}
\newtheorem{cor}[thrm]{Corollary}

\theoremstyle{definition}
\newtheorem{defn}[thrm]{Definition}
\newtheorem{ex}[thrm]{Example}

\newtheorem{note}[thrm]{Note}

\theoremstyle{remark}

\newtheorem{remark}[thrm]{Remark}
\newtheorem{notation}[thrm]{Notation}

\numberwithin{equation}{section}

\newcommand{\tildeQ}{\widetilde{Q^\tau}}
 
\DeclareMathOperator{\Id}{Id}
\DeclareMathOperator{\SL}{SL}
\DeclareMathOperator{\Inn}{Inn}
\DeclareMathOperator{\Aut}{Aut}
\DeclareMathOperator{\GL}{GL}

\title[Cartan and Iwasawa Decompositions for $\SL_2(k)$]{Generalizations of the Cartan and Iwasawa Decompositions for $\SL_2(k)$}
\author{Amanda K. Sutherland}
\address{Department of Mathematics\\North Carolina State University \\ Raleigh, NC 27695}
\email{aksuther@nscu.edu}

\begin{document}

\maketitle
\begin{abstract}
The Cartan and Iwasawa decompositions of real reductive Lie groups play a fundamental role in the representation theory of the groups and their corresponding symmetric spaces. These decompositions are defined by an involution with a compact fixed-point group, called a Cartan involution. For an arbitrary involution, one can consider similar decompositions. We offer a generalization of the Cartan and Iwasawa decompositions for algebraic groups defined over an arbitrary field $k$ and a general involution. 

\end{abstract}

\section{Introduction}

The Cartan decomposition of a real reductive Lie group $G$ factors the group into $HQ$ where $H$ is maximal compact and $Q$ is the symmetric space with respect to the Cartan involution. The Cartan decomposition generalizes the polar decomposition or singular value decomposition of matrices. 
The Iwasawa decomposition of a real reductive Lie group factors the group into its analytical subgroups $HP$ where $H$ is maximal compact and $P$ is a minimal parabolic $\mathbb R$-subgroup. This decomposition results from combining the Cartan decomposition of a semisimple Lie algebra and the root space decomposition of its complexifciation. 
The Cartan and Iwasawa decomposition of real reductive Lie groups plays an important role in representation theory and in the structure of their corresponding real reductive symmetric spaces. The reader is referred to Helgason \cite{helgason} for a more complete description of these decompositions. In \cite{helmwang}, 
a generalization of the notion of a Cartan involution is given and the Cartan and Iwasawa decompositions are generalized to the groups with such an involution.

In this paper, we let $G=\SL_2(\bar k)$ be an algebraic group defined over a field $k$ of characteristic not $2$ and develop a decomposition which resembles a combination of the Cartan and Iwasawa decomposition. This decomposition plays a role in the study of the generalized symmetric spaces of algebraic groups. We extend the notion of the Cartan and Iwasawa decompositions to $G$ defined over other fields. Specifically, we consider the real, rational, and $\frak{p}$-adic numbers, as well as the finite fields. We also generalize the factors of the decompositions by defining them with respect to any involution of the group. 

In section \ref{preliminaries}, we review results and notation needed to prove our main results. In section \ref{extending}, we show $\SL_2(k)$ can be factored to $H_k^\tau \tildeQ U_k$ where $H^\tau$ is the fixed-point group of some involution $\tau$, $\tildeQ=\{g\in G \ \vert \ \tau(g)=g^{-1}\}$ is the extended symmetric space of the involution $\tau$, and $U$ a unipotent subgroup of $G$. In section \ref{structure}, we discuss the structure of the symmetric and extended symmetric spaces. In section \ref{refining}, we analyze our decomposition of $\SL_2(k)$ in more detail and refine it for specific fields and involutions.
In section \ref{char2}, we summarize our results for fields of characteristic $2$. 


\section{Preliminaries}\label{preliminaries}

We borrow most notation from Springer and Borel \cite{springer, boreltits, boreltits2, borel}.

\begin{notation}

Let $k$ be a field of characteristic not equal to $2$ and $\bar k$ the algebraic closure of $k$. We will use $G=\SL_2(\bar k)$ and $G_k=\SL_2(k)$, the $k$-rational points of $G$. In general, for a group $A$ defined over $k$, $A_k$ will denote the $k$-rational points of $A$. 
\end{notation}
\subsection{Automorphisms of $G$}
For $B\in \GL_2(k)$, let $\Inn(B)$ denote the automorphism of $G$ defined by $\Inn(B)(X)=BXB^{-1}$ for all $X\in G$. Let $\Aut(G,G_k)$ denote the group of automorphisms of $G$ which keep $G_k$ invariant. We say $\phi,\theta\in \Aut(G,G_k)$ are isomorphic (over $k$) if there exists a third automorphism $\chi\in \Aut(G,G_k)$ such that $\chi\phi\chi^{-1}=\theta$. This is denoted $\phi\simeq \theta$ when the field $k$ is clear from context.

Combining results from \cite{borel} and \cite{helmwu}, we have the following Lemma. 

\begin{lem}
All automorphisms $\phi\in \Aut(G,G_k)$ are isomorphic over $k$ to $\Inn(A)$ for some $A\in\GL_2(k)$. 
\end{lem}

\subsection{Square classes of $k$}For a field $k$, let $k^*$ denote the product group of non-zero elements from $k$ and $(k^*)^2$ the normal subgroup of squares in $k^*$ defined by $(k^*)^2=\{a^2 \ \vert \ a\in k^*\}$. The quotient group $k^*/(k^*)^2$ is the set of square classes in $k$. 

 From \cite{helmwu,helm} we borrow the following results about automorphisms of order $2$,  called involutions, of $G_k$. 

\begin{thrm}

All involutions $\tau\in \Aut(G,G_k)$ are isomorphic over $k$ to Inn$(B)$, where $B=\left(\begin{smallmatrix} 0 & 1 \\ b & 0\end{smallmatrix} \right)$ for some $b\in k^*$. 
\end{thrm}

\begin{thrm}

Let $M=\left(\begin{smallmatrix} 0 & 1 \\ m & 0 \end{smallmatrix}\right)$ and $N=\left(\begin{smallmatrix} 0 & 1 \\ n & 0 \end{smallmatrix}\right)$ be the matrices corresponding to  $\Inn(M),\Inn(N)\in\Aut(G,G_k)$, respectively. Then $\Inn(M)\simeq \Inn(N)$ if and only if $m$ and $n$ are in the same square class of $k$. 
\end{thrm}

\begin{cor}\label{numberofinvs} The number of isomorphy classes of involutions of $G$ which keep $G_k$ invariant is $\vert k^*/(k^*)^2\vert$.
\end{cor}

\begin{notation}
Let $m$ be a representative of the square class of $\overline m$ in $k^*/(k^*)^2$. We will use $\tau_m$ to denote the involution $\Inn(M)$ of $G$ with $M=\left(\begin{smallmatrix} 0 & 1 \\ m & 0\end{smallmatrix}\right)$.
\end{notation} 
\begin{remark}
For all involutions $\tau\in\Aut(G,G_k)$, we can assume $\tau\simeq \tau_m$ where $m\in k^*$ is the representative of the square class $\overline m$. For the class of squares, we use $\tau_1$.
\end{remark}

\subsection{Fixed-point group of an automorphism} Let $\mathcal G$ be a group and $\phi$ an automorphism of $\mathcal G$. Denote $H^\phi$ as the fixed-point group of $\phi$ in $\mathcal G$.

\begin{ex}
For  $G$ with the involution $\tau_m$, the fixed point group is $H^{\tau_m}$.
\[H_k^{\tau_m}=\left\{ \begin{pmatrix} a & b \\ mb & a \end{pmatrix}  \ \bigg \vert \ a,b\in k, \ a^2-mb^2=1\right\}\]

\end{ex}

\begin{defn} For a group $\mathcal G$ with the involution $\tau$, the \emph{symmetric space} is defined as $Q^\tau=\{g\tau(g)^{-1} \ \vert \ g \in \mathcal G\}$, and the \emph{extended symmetric space} is defined as $\tildeQ=\{g\in \mathcal G \ \vert \ \tau(g)=g^{-1}\}$.
\end{defn}

\begin{remark}\label{G/H}
For a group $\mathcal G$ with the involution $\tau$, the symmetric space is isomorphic to $\mathcal G/H^\tau$.

\end{remark}

\begin{ex}For $G$ with the involution $\tau_m$, the extended symmetric space is $\widetilde{Q^{\tau_m}}$.

\[\widetilde{Q^{\tau_m}}=\left\{\begin{pmatrix} a & b \\ -mb & c \end{pmatrix} \ \bigg\vert \ a,b,c\in k, \ ac+mb^2=1\right\}\]
\end{ex}

\begin{remark}
An element is $\tau$-split if it is sent to its inverse under $\tau$. A subset is $\tau$-split if all of its elements are $\tau$-split, i.e. $\tildeQ$ is $\tau$-split by definition. A torus is $(\tau,k)$-split if it is both $\tau$-split and $k$-split.
\end{remark}
\begin{lem} The symmetric space of a connected group is connected. Furthermore, $(\tildeQ)^\circ=Q^\tau$, where $(\tildeQ)^\circ$  denotes the connected component of the extended symmetric space containing the identity. 
\end{lem}
\begin{proof}
Consider a connected group $\mathcal G$ defined over $k$, a field with a topology. Then $Q^\tau$ is connected because it is the image of the continuous mapping defined by $g\mapsto g\tau(g)^{-1}$ for $g\in \mathcal G$. Since $(\Id)\tau(\Id)^{-1}=\Id$, the identity matrix is always contained in the symmetric space. Therefore, the symmetric space is the connected component of the extended symmetric space containing the identity, $Q^\tau=(\tildeQ)^{\circ}$.  
\end{proof}

\begin{defn}
An involution $\tau$ of a group is a \emph{generalized Cartan involution} if the fixed-point group of $\tau$ is $k$-anisotropic. 
\end{defn}

\begin{remark}

A subgroup $A$ is $k$-isotropic if it contains a $k$-split torus. Otherwise, $A$ is $k$-anisotropic.
\end{remark}

\begin{remark}
If $k=\mathbb R$, this is the regular Cartan involution. $\mathbb R$-anisotropic is equivalent to compact. By abuse of notation, we will refer to a generalized Cartan involution as a Cartan involution.

\end{remark}

\begin{ex}\label{CartanSL2}
Consider $G$ defined over $k=\mathbb R$. The square classes of $\mathbb R$ are represented by $\{ 1, -1\}$. Up to isomorphy over $\mathbb R$, there are two involutions of $G$ which keep $G_{\mathbb R}$ invariant, namely $\tau_1$ and $\tau_{-1}$. 

The fixed-point group of $\tau_{-1}$ is the special orthogonal group SO$(2)$ which is compact and hence $\tau_{-1}$ is a Cartan involution. The symmetric space of  $\tau_{-1}$ is the set of positive definite symmetric matrices, while the extended symmetric space is the set of symmetric matrices. 

For the involution $\tau_1$, the fixed-point group is the subgroup SO$(1,1)$.
\end{ex}

The following result gives us the Cartan decomposition.
\begin{thrm}\label{cartan}

Let $\mathcal G$ be a real semisimple Lie group and $\theta$ a Cartan involution of $\mathcal G$. Define $Q$ and $H$ to be the symmetric space and fixed-point group with respect to $\theta$, respectively, and $A$ a maximal $(\theta,k)$-split torus in $\mathcal G$. Then $\theta$ induces the following equivalent Cartan decompositions:

\[ \mathcal G=HQ=H^\circ Q = HAH = H^\circ A^ \circ\]

where $H^\circ$ denotes the connected component of $H$ containing the identity. 

\end{thrm}

The following result gives us the Iwasawa decomposition.
\begin{thrm}\label{iwasawa}

Let $\mathcal G$ be a real semisimple Lie group and $\theta$ a Cartan involution of $ G$. Let $H$ be the fixed-point group and $P$ a minimal parabolic $\mathbb R$-subgroup. Then $\theta$ induces the Iwasawa decomposition:
\[ \mathcal G=HP\]
\end{thrm}

\begin{remark}\label{iwasawaremark}
For a $k$-split group, as is the case with $G$, we can write $G_{\mathbb R}=H_{\mathbb R}A_{\mathbb R}U_{\mathbb R}$, where $A$ is the maximal $(\theta,k)$-split torus and $U$ a maximal unipotent subgroup defined over $k$. In fact, $P_{\mathbb R}=Z_{G_{\mathbb R}}(A)U_{\mathbb R}=A_{\mathbb R}U_{\mathbb R}$ since $A$ is a maximal torus.
\end{remark}

From \cite{helmwang}, we have a condition which equates the Iwasawa and Cartan decompositions. 
\begin{thrm}

Let $\mathcal G$ be a real Lie group defined over a field $k$ as in \cite{helmwang} and $\theta$ a generalized Cartan involution of $G$. If $(k^*)^2=(k^*)^4$, then the following decompositions are equivalent:
\[ \mathcal G_k=H_k^\circ A_k H_k^\circ=H_kA_kH_k=H_k^\circ Q=H_kQ=H_k^\circ A_k U_k=H_kA_kU_k\]
where $H,A,Q,$ and $U$ are as defined in Theorems \ref{cartan} and Remark \ref{iwasawaremark}.
\end{thrm}

\begin{remark}
The criteria for a generalized Cartan involution in \cite{helmwang} is much stronger than in this paper. The additional conditions in \cite{helmwang} guarantee the existence of (a generalization of) a Cartan and Iwasawa decomposition. 

\end{remark}

\section{Generalizing the Decompositions to Algebraic Groups}\label{extending}

As previously discussed,
the Cartan and Iwasawa decompositions are defined for real semisimple Lie groups when paired with a Cartan involution. In general, for any field $k$ with a general involution $\tau$, the set $H_k^\tau Q^\tau$ is contained in, but not equal to, $G_k$.

\begin{ex}
Let $G$ be defined over $k=\mathbb R$ and $\tau=\tau_1$ the involution of $G$.

Consider $g=\left(\begin{smallmatrix} 1 & 2(\sqrt{5}-3)^{-1} \\ \frac{1}{2}(\sqrt{5}-3) & 2 \end{smallmatrix}\right)\in G_{\mathbb R}$. 
If $g\in H_{\mathbb R}^\tau\tildeQ$, then there exists $h=\left(\begin{smallmatrix} a & b \\ b & a \end{smallmatrix}\right)\in H_{\mathbb R}^\tau$ such that $h^{-1}g\in \tildeQ$. Computing $\tau(h^{-1}g)=(h^{-1}g)^{-1}$ as in \eqref{hg}, implies $a=b$, hence $h\not\in H_{\mathbb R}^\tau$. Therefore $g\not\in H_{\mathbb R}^\tau\tildeQ$. Borrowing Lemma \ref{QinQtilde} from later in the paper stating $Q^\tau\subset \tildeQ$, we see $g\not\in H_{\mathbb R}^\tau Q^\tau$, thus the traditional Cartan decomposition does not hold. 

\begin{equation}\label{hg}
\tau(h^{-1}g)=
\begin{pmatrix} \frac{2b}{\sqrt{5}-3}+2a & b+a\left(\frac{\sqrt{5}-3}{2}\right) \\ \frac{2a}{\sqrt{5}-3}+2b & a+b\left(\frac{\sqrt{5}-3}{2}\right)\end{pmatrix} = \begin{pmatrix} \frac{2b}{\sqrt{5}-3}+2a & -2b-\frac{2a}{\sqrt{5}-3} \\ \left(\frac{3-\sqrt{5}}{2}\right)a-b & a+b\left(\frac{\sqrt{5}-3}{2}\right)\end{pmatrix}=(h^{-1}g)^{-1}
\end{equation}

\end{ex}

To account for the missing elements, we introduce the unipotent subgroup $U$ of $G$ consisting of upper triangular matrices with ones on the diagonal. 

\[U_k=\left\{\begin{pmatrix} 1 & u_1 \\ 0 & 1\end{pmatrix} \bigg\vert u_1 \in k\right\}\]

With the addition of the new subgroup, we have the following result.

\begin{thrm}\label{HQR}

For $G$ with the involution $\tau$ of $G$, 
$G_k=H_k^\tau \tildeQ U_k$,
where $H^\tau, \tildeQ$, and $U$ are the fixed-point group, extended symmetric space, and unipotent subgroup, respectively, of $G$.

\end{thrm}

 \begin{remark}This decomposition serves as a generalization of both the Cartan and Iwasawa decompositions. It contains the fixed-point group and symmetric space similar to the Cartan decomposition. Additionally, because the maximal $(\theta,k)$-split torus $A$ is contained in $\tildeQ$ and we have a unipotent subgroup, it generalizes the Iwasawa decomposition. \end{remark}
 
 To prove Theorem \ref{HQR}, we use the Bruhat Decomposition as in \cite{boreltits2}.
 
 \begin{thrm}
 
For an algebraic group $\mathcal G$, let $P$ be a minimal parabolic $k$-subgroup of $G$, $A$ a maximal $k$-split torus in $P$, and $W(A)$ the Weyl group of $A$.
Then $G_k$ decomposes as the disjoint union of the double cosets of $P_k$ parameterized by $W(A)$. 
\[G_k=\bigcupdot_{\omega\in W(A)}P_k\omega P_k\]
\end{thrm}

\begin{remark}[Bruhat Decomposition of $G$]
For a $k$-split group, $P=B$ is a Borel subgroup and $A=T$ is a maximal torus. For $G_k$, the Bruhat decomposition is $\displaystyle G_k=\bigcupdot _{\omega\in W(T)}B_k\omega B_k$.

Let the maximal torus $T$ be the subgroup of diagonal matrices in $G$, and  the Borel subgroup $B\supset T$ the upper triangular matrices in $G$.

\begin{equation}\label{bandt} B_k=\left\{\begin{pmatrix} x & y \\ 0 & z \end{pmatrix} \bigg \vert x,y,z\in k, xz=1 \right\} \quad T_k=\left\{\begin{pmatrix} a & 0 \\ 0 & a^{-1} \end{pmatrix} \bigg \vert a\in k^*\right\}\end{equation}

Let $\Id$ be the $2\times 2$ identity matrix, then we can define the Weyl group $W(T)$ and Bruhat decomposition of $G_k$.

\begin{equation}\label{bruhat}W(T)=\left\{ \Id,\begin{pmatrix} 0 & 1 \\ -1 & 0 \end{pmatrix}\right\} \quad G_k=B_k\bigcupdot B_k\begin{pmatrix} 0 & 1 \\ -1 & 0 \end{pmatrix}B_k\end{equation}
\end{remark}

\begin{remark}\label{b=tu}
 For the Borel subgroup $B$, we can write $B=T U$, where $T$ is the $k$-split maximal torus and $U$ is the unipotent radical.
\end{remark}

\begin{lem}\label{taustable}

Let $\tau$ be an involution of $G$ and $T$ the $k$-split maximal torus of diagonal matrices. Then $T$ is invariant under $\tau$ and is maximal $(\tau,k)$-split. 
\end{lem}

\begin{proof}
Let $\tau=\tau_m$ be an involution of $G$ and $t=\left(\begin{smallmatrix} a & 0 \\ 0 & a^{-1} \end{smallmatrix}\right)\in T$. Then $T$ is $\tau$-split, thus invariant under $\tau$.
\[\tau(t)=\begin{pmatrix} 0 & 1 \\ m & 0\end{pmatrix} \begin{pmatrix}a & 0 \\ 0 & a^{-1}\end{pmatrix} \begin{pmatrix} 0 & m^{-1} \\ 1 & 0 \end{pmatrix}=\begin{pmatrix} a^{-1} & 0 \\ 0 & a \end{pmatrix}=t^{-1}\in T\]
\end{proof}

\begin{remark} A subgroup which is invariant under an automorphism $\phi$ is said to be $\phi$-stable. \end{remark}

\begin{lem}\label{F3}
Let $G$ be defined over $k$ and $\tau$ an involution of $G$. If $H^\tau_k=\{\pm\Id\}$, then $k\simeq\mathbb F_3$.
\end{lem}

\begin{proof}
From \cite{helmwang}, the fixed-point group on an involution is always reductive. Thus for $G$, the fixed-point group is a torus. Because $H\simeq \overline{ k^*}$, if $\vert H\vert=2$, it must be that $k\simeq \mathbb F_3$.  
\end{proof}

\begin{remark}
For $G$ defined over $k=\mathbb F_3$, $H^{\tau_m}_k=\{\pm\Id\}$ only for $m\in (k^*)^2$.
\end{remark}

\begin{proof}[Proof of Theorem \ref{HQR}] 

Let $\tau$ be an involution of $G$. Because $H_k^\tau, \tildeQ,$ and $U_k$ are contained in $G_k$, $H_k^\tau \tildeQ U_k \subset G_k$ is clear. 
We will show the reverse containment, $G_k\subset H_k^\tau \tildeQ U_k$ using the equivalent statement \eqref{equiv}, replacing $G_k$ with its Bruhat decomposition as in \eqref{bruhat}.
\begin{equation}\label{equiv}B_k\bigcupdot B_k\begin{pmatrix} 0 & 1 \\ -1 & 0 \end{pmatrix}B_k\subset H_k^\tau \tildeQ U_k\end{equation}

First, consider $g\in B_k$. By Remark \ref{b=tu}, $u^{-1}g=t$ for some $u^{-1}\in U_k$ and $t\in T_k$. By Lemma \ref{taustable}, $u^{-1}g=t$ is $\tau$-split, hence $u^{-1}g\in \tildeQ$. Therefore $g\in \tildeQ\subset H_k^\tau\tildeQ U_k$.

Second, consider $g\in B_k\left(\begin{smallmatrix} 0 & 1 \\ -1 & 0 \end{smallmatrix}\right)B_k$. Then for some $a,b\in k^*$ and $\alpha,\beta\in k$, we rewrite $g$ as in \eqref{bwb}.
 \begin{equation}\label{bwb}g=\begin{pmatrix}a & \alpha \\ 0 & a^{-1} \end{pmatrix}\begin{pmatrix} 0 & 1 \\ -1 & 0 \end{pmatrix} \begin{pmatrix}b & \beta \\ 0 & b^{-1}\end{pmatrix}\end{equation}

If $\alpha\neq 0$, let $u=\left(\begin{smallmatrix} 1 & \frac{ma^2-b^2-m\alpha\beta ab}{m\alpha ab^2} \\ 0 & 1\end{smallmatrix}\right)$.
Then $gu\in\tildeQ$ by \eqref{guinQ}. Therefore, $g\in H_k^\tau\tildeQ U_k$. 

\begin{equation}\label{guinQ} 
\tau(gu)=\begin{pmatrix} -\frac{a^2m-b^2}{a^2bm\alpha} & -\frac{b}{ma} \\ \frac{b}{a} & -\alpha b\end{pmatrix} =(gu)^{-1}
 \end{equation}

If $\alpha=0$, let $h=\left(\begin{smallmatrix} a_1 & b_1 \\ mb_1 & a_1 \end{smallmatrix}\right) \in H_k^\tau \setminus\{\pm\Id\}$, and $u=\left(\begin{smallmatrix} 1 & \frac{ma_1a^2-b^2a_1-mb_1\beta b}{mb^2b_1} \\ 0 & 1 \end{smallmatrix}\right)$. Then $hgu\in \tildeQ$ by \eqref{hgu}. Therefore $g\in H_k^\tau\tildeQ U_k$.

\begin{equation}\label{hgu}
\tau(hgu)=\begin{pmatrix}  -\frac{b_1b}{a} & \frac{a_1b}{am} \\ -\frac{a_1b}{a} & \frac{a^2m^2b_1^2-ma^2a_1^2+b^2a_1^2}{abmb_1}\end{pmatrix} =(hgu)^{-1}
\end{equation}

For $k=\mathbb F_3$, one can easily verify this results holds although $H=\{\pm \Id\}$.
\end{proof}

\begin{remark}
For a general field $k$ and involution $\tau$, $G_k\neq H^\tau_k Q^\tau U_k$ and thus expanding to the extended symmetric space is necessary. Further on, we give cases in which the symmetric space will suffice. 
\end{remark}

\begin{remark}\label{orbits} Let $G$ be defined over a field $k$ and $\tau$ an involution of $G$. 
The $(H^\tau_k \times U_k)$-orbits on $G_k$ are defined by $(h,u)\bullet g:=hgu$ for $h\in H^\tau_k, u\in U_k,$ and $g\in G_k$. By Theoren \ref{HQR}, we can choose orbit representatives in $\tildeQ$. Similarly, the twisted $U_k$-orbits on $\tildeQ$ are defined by the twisted action $u*q:=u^{-1}q\tau(u)$ for $u\in U_k$ and $q\in Q^\tau$.  
By Remark \ref{G/H}, $Q^\tau\simeq G_k/H^\tau_k$, thus the $U_k$-orbits on $Q^\tau$ are in bijective correspondence with the $(H^\tau_k \times U_k)$-orbits on $G_k$ if and only if $H^\tau_kQ^\tau U_k = H^\tau_k\tildeQ U_k$.

\end{remark}

\begin{ex}
Let $G$ be defined over $k=\mathbb Q$ and $\tau=\tau_{-1}$ the involution of $G$. For convenience, we will use $Q^\tau=\{g^{-1}\tau(g)\vert g\in G_k\}$. 
Let the $(H^\tau_k \times U_k)$-orbits on $G_k$ and twisted $U_k$-orbits on $Q^\tau$ be defined as in Remark \ref{orbits}.
Then consider the following
map from $U_k*Q^\tau$ to $(H_k^\tau\times U_k)\bullet G_k$, where $q=g^{-1}\tau(g)$.
\[U_k*q \mapsto (H^\tau_k\times U_k)\bullet g\]

For $q \in Q^\tau$, assume there exists $g_1,g_2\in G_k$ such that $q=g_1^{-1}\tau(g_1)=g_2^{-1}\tau(g_2)$. Then $g_1=hg_2$ for some $h\in H^\tau_k$ by \eqref{g1g2}.
\begin{equation}\label{g1g2}
g_1^{-1}\tau(g_1)=g_2^{-1}\tau(g_2) \Rightarrow \tau(g_1g_2^{-1})=g_1g_2^{-1}\Rightarrow g_1g_2^{-1}\in H^\tau_k
\end{equation}
The map is well-defined because it is independent of coset representative and surjective by definition of $Q^\tau$. 
We may also reverse the map by \eqref{reversemap}.  
\begin{equation}\label{reversemap} (H^\tau_k\times U_k)\bullet g \mapsto g^{-1}\tau(g) \end{equation}
By Theorem \ref{HQR}, let $g=hqu$, $q\in\tildeQ$. Then $(H^\tau_k \times U_k)\bullet g$ maps to 
\[(g)^{-1}\tau(g)=(hqu)^{-1}\tau(hqu)=u^{-1}q^{-1}h^{-1}\tau(h)\tau(q)\tau(u)=u^{1}q^{-1}\tau(q)u=u*q_0\in U_k*q_0\]

for some $q_0=q^{-2}\in Q^\tau$. Over $k=\mathbb Q$, this map is not surjective because not all elements of $Q^\tau$ can be written as $q^{-2}$ for some $q\in \tildeQ$. 

From \cite[Proposition 6.6]{helmwang}, the $U_k$-orbits on $Q^\tau$ can always be represented by an element from the normalizer in $G$ of a $\tau$-stable maximal $k$-split torus $A$, $N_G(A)$. In this case, $N_G(A)\cap Q^\tau$ is the set of diagonal elements in $G_k$. 
Furthermore, the action of $(H^\tau_k\times U_k)$ on $\lambda \in N_G(A)\cap Q^\tau$ can not map $\lambda$ to another element $\mu\in N_G(A)\cap Q^\tau, \mu\neq -\lambda$.
\[(H^\tau_k\times U_k)\bullet\lambda=h\lambda u=\begin{pmatrix}a & b \\ -b & a \end{pmatrix}\begin{pmatrix} \lambda_1 & 0 \\ 0 & \lambda_1^{-1} \end{pmatrix} \begin{pmatrix} 1 & u_1 \\ 0 & 1\end{pmatrix}=\begin{pmatrix} a\lambda_1& u_1a\lambda_1+ b\lambda_1^{-1} \\-b\lambda_1 & -b\lambda_1 u_1+a\lambda_1^{-1}\end{pmatrix} \]

\begin{equation}\label{normelements}
\begin{pmatrix} a\lambda_1& u_1a\lambda_1+ b\lambda_1^{-1} \\-b\lambda_1 & -b\lambda_1 u_1+a\lambda_1^{-1}\end{pmatrix}=\begin{pmatrix} \mu_1 & 0 \\ 0 & \mu_1^{-1}\end{pmatrix}
\end{equation}

Solving \ref{normelements}, we get $h=\pm\Id$ and $u=\Id$. 

If we let $q^{-1}=\left(\begin{smallmatrix} x & y \\ y  & z \end{smallmatrix}\right)\in \tildeQ$, then the only $U_k$-orbits on $Q^\tau$ which correspond to the $(H^\tau_k\times U_k)$-orbits on $G_k$ are the ones whose representative in $N_G(A)$ is of the form $q_0=\left (\begin{smallmatrix} x^2 +y^2 & 0 \\ 0 & y^2+z^2\end{smallmatrix}\right)$. Let $g=\left(\begin{smallmatrix} \lambda_1 & 0 \\ 0 & \lambda_1^{-1}\end{smallmatrix}\right)\in G_k$ such that $\lambda_1>0$ and $\lambda_1$ is not the sum of two squares, then $(H^\tau_k \times U_k)\bullet g$ can not be obtained as a $U_k$-orbit on $Q^\tau$. Hence $G\neq H^\tau_k Q^\tau U_k$. 

\end{ex}

\section{Symmetric Space and Extended Symmetric Spaces}\label{structure}

To understand the structure of the symmetric and extended symmetric spaces of $G$ for any field and involution, we will analyze the relationship between the spaces and then their semisimplicity.  
\subsection{Relationship between the symmetric and extended symmetric Spaces}
\begin{lem}\label{QinQtilde}
For a group $\mathcal G$ with an involution $\tau$, 
the symmetric space is contained within the extended symmetric space. i.e. $Q^\tau \subset \tildeQ$. 
\end{lem}

\begin{proof}
Let $g\tau(g)^{-1}\in Q^\tau$ for some $g\in \mathcal G$, then $g\tau(g)^{-1}\in \tildeQ$.
 \[\tau(q)=\tau(g\tau(g)^{-1})=\tau(g)\tau^2(g)^{-1}=\tau(g)g^{-1}=q^{-1}\]
 \end{proof}
 
 Example \ref{CartanSL2} demonstrates that the symmetric space and extended symmetric space are not equivalent in general. We will determine for which cases we get equality.  

 \begin{thrm}\label{kbartheorem} Let $G$ be defined over $k=\bar k$ and $\tau=\tau_1$ the involution of $G$. Then the extended symmetric space is equivalent to the symmetric space. \end{thrm}
 \begin{proof}
Let $q=\left(\begin{smallmatrix} a & b \\ -b & c \end{smallmatrix}\right)\in \tildeQ$. For $q$ to be in the symmetric space, we need $g\in G$ such that $g\tau(g)^{-1}=q$. Depending on the value of $c$, choose $g\in G$ according to Table \ref{kbar}. Choosing the appropriate $g$ will yield $g\tau(g)^{-1} = q \in Q^\tau$. The reverse containment follows from Lemma 4.1.
\begin{table}
\centering
\caption{$g\in G$ such that $q=g\tau(g)^{-1}\in Q^\tau$ for $k=\bar k$}
\begin{tabular}{| c |c |c| }
\hline
  $c$ & $b$ & $g\in G$ \\
  \hline
  $c\neq 0$ & - & $\begin{pmatrix} \frac{1}{\sqrt{c}} & \frac{b}{\sqrt{c}} \\ 0 & \sqrt{c} \end{pmatrix}$ \\
  \hline
  $c=0$ & b=1 & $\begin{pmatrix} 0 & \sqrt{-a} \\ \frac{\sqrt{-a}}{a} & -\frac{b\sqrt{-a}}{a} \end{pmatrix}$ \\
  \hline 
  $c=0$ & $b=-1$ & $\begin{smallmatrix} 1 & -1 \\ \frac{1-c}{2} & \frac{c+1}{2}\end{smallmatrix}$\\
  \hline
\end{tabular}

\label{kbar}
\end{table}

 \end{proof}
 
 \begin{thrm}\label{real}
 
 Let $G$ be defined over $k=\mathbb R$ and $\tau=\tau_1$ the involution of $G$. Then the extended symmetric space is equivalent to the symmetric space.
\end{thrm}

\begin{proof}
Let $q=\left(\begin{smallmatrix} a & b \\ -b & c \end{smallmatrix}\right)\in \tildeQ$. As in the proof of Theorem \ref{kbartheorem}, we need $g\in G_{\mathbb R}$ such that $g\tau(g)^{-1}=q$. Depending on the value of $a$, choose $g\in G_{\mathbb R}$ according to Table \ref{Rtable}. Choosing the appropriate $g$ will yield $g\tau(g)^{-1}=q\in Q^\tau$. The reverse containment follows from Lemma \ref{QinQtilde}.

\begin{table}
\caption{$g\in G_{\mathbb R}$ such that $q=g\tau(g)^{-1}\in Q^\tau$ for $k=\mathbb R$}
\centering
\begin{tabular}{| c |c |c| }
\hline
  $a$ & $b$ & $g\in G_{\mathbb R}$ \\
  \hline
  $a>0$ & - & $\begin{pmatrix} \sqrt{a} & 0 \\ -\frac{b}{\sqrt{a}}& \frac{1}{\sqrt{a}}  \end{pmatrix}$ \\
  \hline
  $a<0$ & - & $\begin{pmatrix} 0 & \sqrt{-a} \\ \frac{\sqrt{-a}}{a} & -\frac{b\sqrt{-a}}{a} \end{pmatrix}$ \\
  \hline
  $a=0$ & $b=1$ & $\begin{pmatrix} 1 & 1 \\ \frac{c-1}{2} & \frac{c+1}{2} \end{pmatrix}$\\
  \hline 
  $a=0$ & $b=-1$ & $\begin{pmatrix} 1 & -1 \\ \frac{1-c}{2} & \frac{c+1}{2}\end{pmatrix}$\\
  \hline
\end{tabular}
\label{Rtable}
\end{table}
\end{proof}

In \cite{Buelletal}, the structure of the symmetric space of $G$ defined over $k=\mathbb F_q$ is analyzed, including the following result. 
\begin{thrm}\label{symspfinite}

Let $G$ be defined over $k=\mathbb F_q$, with characteristic of $k$ not $2$, and $\tau$ an involution of $G$. Then the symmetric space is equivalent to extended symmetric space.
\end{thrm}

\subsection{Relationship between the symmetric and extended symmetric spaces for $G$ defined over a $\frak p$-adic field}
\hfill

We now consider $G$ defined over $k=\mathbb Q_p$ with the involution $\tau=\tau_1$. For $q=\left(\begin{smallmatrix} a & b \\ -b & c \end{smallmatrix}\right)\in \tildeQ$, we must show there exists $g\in G_{\mathbb Q_p}$ such that $q=g\tau(g)^{-1}$. To find such $g\in G_{\mathbb Q_p}$, we solve for $x,y,z,w\in \mathbb Q_p$ such that the following equations hold.
\begin{equation}\label{det} 
xw-yz=1
\end{equation}
\begin{equation}\label{a}
x^2-z^2=a
\end{equation}
\begin{equation}\label{c}
w^2-z^2=c
\end{equation}
\begin{equation}\label{b}
yw-xz=b
\end{equation}

Using Hilbert's symbol, we obtain solutions to \eqref{a} and (\ref{c}) in $\mathbb Q_p$. 
\begin{defn} For $a,b\in \mathbb Q_p$, the Hilbert symbol is defined as
\[(a,b)_p=\left\{\begin{array}{lr} 1 & ax^2+by^2-z^2 \text{ is isotropic}
\\ -1 & \text{ otherwise }
\end{array} \right.\]
\end{defn}

\begin{remark}
The polynomial $ax^2+bx^2-z^2$ is isotropic over $\mathbb Q_p$ if there exists non-trivial $(x,y,z)$ in $\mathbb Q_p^3$ such that $ax^2+bx^2-z^2=0$.

\end{remark}

\begin{prop}[Properties of the Hilbert Symbol]\label{Hilbertsymbolprops}
\hfill

For all $a,b \in \mathbb Q_p$, we have

\begin{enumerate}
\item $(a,b)_p=(b,a)_p$
\item $(a,b)_p=1$ if $a\in (\mathbb Q_p^*)^2$
\item $(a,-a)_p=1$
\end{enumerate}
\end{prop}

The Hilbert symbol shows that \eqref{a} has a solution because the equivalent equation \eqref{hilbertequation}, is isotropic over $k=\mathbb Q_p$.
\begin{equation}\label{hilbertequation} x^2-z^2-y^2a=\frac{1}{a}x^2+\left(-\frac{1}{a}\right)z^2-y^2\end{equation}

By Proposition \ref{Hilbertsymbolprops}, $(\frac{1}{a},-\frac{1}{a})_p=1$. Therefore \eqref{a} has  a non-trivial solution $(x,z,y)$ in $\mathbb Q_p^3$. We then scale our solution such that $y=1$. Similar calculations show \eqref{b} has a solution.

For a simultaneous solution to (\ref{det})-(\ref{b}), let 
\[x=\frac{-b+\alpha w}{\beta}, y=\alpha, z=\beta\]
where \[\alpha=\frac{wb \pm \sqrt{w^2-c}}{c}, \beta=\pm \sqrt{w^2-c}.\]

For $\alpha, \beta \in \mathbb Q_p$, we must verify there exists $\beta \in \mathbb Q_p$ for all $c\in \mathbb Q_p$. 
\begin{equation}\label{beta}
\beta^2=w^2-c
\end{equation}

Using the Hilbert symbol, (\ref{beta}) corresponds to $(1,-1)_p=1$ and therefore $\beta\in \mathbb Q_p$ exists.
Depending on the value of $c$, choose $g\in G_{\mathbb Q_p}$ according to Table \ref{Qp}. Choosing the appropriate $g$ will yield $g\tau(g)^{-1}=q\in Q^\tau$. Because the reverse containment is clear by Lemma \ref{QinQtilde}, we have the following result.

\begin{table}
\caption{$g\in G_{\mathbb Q_p}$ such that $q=g\tau(g)^{-1}\in Q^\tau$ for $k=\mathbb Q_p$}
\centering
\begin{tabular}{| c |c |c| }
\hline
  $c$ & $b$ & $g\in G_{\mathbb Q_p}$ \\
  \hline
  $c\neq 0$ & - & $\begin{pmatrix} \frac{-bc+w^2b+w\sqrt{w^2-c}}{c\sqrt{w^2-c}} & \frac{wb+\sqrt{w^2-c}}{c} \\ \sqrt{w^2-c} & w \end{pmatrix}$\\
  \hline
  $c=0$ & b=1 & $\begin{pmatrix} -\frac{1}{2} & -\frac{1}{2} \\ 1 & -1 \end{pmatrix}$ \\
  \hline
  $c=0$ & $b=-1$ & $\begin{pmatrix} -1 & 1 \\ -\frac{1}{2} & -\frac{1}{2} \end{pmatrix}$\\
  \hline
\end{tabular}
\label{Qp}
\end{table}

\begin{thrm}
Let $G$ be defined over $k=\mathbb Q_p$ with $p\neq 2$, and $\tau=\tau_1$ the involution of $G$. Then the extended symmetric space is equivalent to the symmetric space.\end{thrm}

\begin{note}By Corollary \ref{numberofinvs}, there are four classes of involutions of $\Aut(G,G_{\mathbb Q_p})$ because $\left\vert \mathbb Q_p^* / (\mathbb Q_p^*)^2\right \vert=4$. These classes are represented by $\{\tau_1, \tau_p,\tau_{N_p}, \tau_{pN_p}\}$, where $N_p$ is a non-square in $\mathbb Q_p$ not in the same square class as $p$. For $\tau=\tau_m$ with $m\neq 1$, we can show by example that the extended symmetric space and symmetric space are not equivalent.

For each involution, we consider the cases when $p\equiv 1 \mod 4$ and $p\equiv 3 \mod 4$ separately. The key difference is that, as in the finite fields, $-1$ is a square when $p\equiv1\mod 4$ and is $-1$ not a square when $p\equiv 3 \mod 4$. When $p\equiv 1 \mod  4$, we use the square class representatives $\mathbb Q^*_p/(\mathbb Q_p^*)^2=\{1,p,N_p,pN_p\}$. When $p\equiv 3\mod 4$, we use the square class representatives $\mathbb Q_p^*/(\mathbb Q_p^*)^2=\{1,p,-1,-p\}$. 
\end{note}

\begin{ex}\label{mp_p3}
Let $p=3$ and consider the involution $\tau=\tau_3$. The extended symmetric space is determined as follows.
 \[\tildeQ=\left\{ \begin{pmatrix} a & b \\ -3b & c\end{pmatrix}\bigg \vert a,b, c \in \mathbb Q_p \text{ and }ac+3b^2=1 \right\}\]

Let $q=\left(\begin{smallmatrix} \frac{1}{3} & 0 \\ 0 & 3 \end{smallmatrix}\right) \in \tildeQ$. 
Then $g=g\tau(g)^{-1}$ implies $g$ is $g_1$ or $g_2$, where $\alpha=\sqrt{3d^2-9}$.
\[g_1=\begin{pmatrix} \frac{\sqrt{3}}{3} & 0 \\ 0 & \sqrt{3} \end{pmatrix}, \quad g_2=\begin{pmatrix} \frac{1}{3}d & \pm\frac{1}{9} \alpha \\ \pm\alpha &d\end{pmatrix}\]
Because $3\not\in(\mathbb Q_3^*)^2$, $g_1 \not\in G_{\mathbb Q_p}$.
 For $g_2$, $\alpha=\sqrt{3d^2-9}=\sqrt{3}\sqrt{d^2-3}\in \mathbb Q_p$ if only if $d^2-3=3$, which implies $d=\sqrt{6}$. Because $6$ is in the square class $pN_p=-3$ and is not a square, $d\not\in \mathbb Q_p$.
Thus there is no $g\in G_{\mathbb Q_p}$ such that $q=g\tau(g)^{-1}$ and the symmetric space and extended symmetric space are not equivalent for the involution $\tau=\tau_p$ when $p\equiv 3\mod 4.$\end{ex}

\begin{remark} Using similar calculations, one can show there exists $q\in \tildeQ$ such that $q\not \in Q^\tau$ in the following cases: $\tau=\tau_p$ with $p\equiv 1 \mod 4, \tau=\tau_{N_p}$ with $p\equiv 1 \mod 4$ and $p\equiv 3 \mod 4,$ and $\tau=\tau_{pN_p}$ with $p\equiv 1 \mod 4$ and $p\equiv 3 \mod 4$. It is left to the reader to find the specific example.  \end{remark} 

These examples serve to show for $G$ defined over $k=\mathbb Q_p$ with the involution $\tau$, $\tildeQ=Q^\tau$ only when $\tau=\tau_1$.

\begin{thrm}[Strong Hasse Principle]\label{hasse}

Let $f$ be a regular quadratic form over $\mathbb Q$. Then $f$ is isotropic over $\mathbb Q$ if and only if $f$ is isotropic over $\mathbb Q_p$ for all $p$, including $p=\infty$. 

\end{thrm}

By Theorem \ref{hasse}, \eqref{det}-\eqref{b} have a simultaneous solution over $k=\mathbb R=\mathbb Q_{\infty}$. 
Applying Theorem \ref{hasse}, \eqref{det}-\eqref{b} must also have a solution over $k=\mathbb Q$, yielding the following result.

 \begin{thrm}\label{rational}

 Let $G$ be defined over $k=\mathbb Q$ and $\tau=\tau_1$ the involution of $G$. Then the extended symmetric space is equivalent to the symmetric space.  \end{thrm}

\begin{cor}\label{tau1}

For $G$ defined over $k$ and $\tau=\tau_1$ the involution of $G$, the extended symmetric space and symmetric space are equivalent over the following fields:
\begin{enumerate}[label=\roman*.]
\item $k=\bar k$
\item $k=\mathbb R$
\item $k=\mathbb F_q$
\item $k=\mathbb Q_p$
\item $k=\mathbb Q$
\end{enumerate}
\end{cor}

\begin{cor}
For $G$ defined over a field listed in Corollary \ref{tau1} with the involution $\tau=\tau_1$ or $G$ defined over $k=\mathbb F_q$ with any involution $\tau$, the decomposition in Theorem \ref{HQR} can be simplified to $G_k=H_kQU_k$. 

\end{cor}

\subsection{Semisimplicity of the symmetric and extended symmetric spaces}
\hfill

From \cite{helmwang} and \cite{beunpaper}, respectively, we have the following results. 
\begin{thrm}\label{helmwang}
 Let $k$ be a field with characteristic zero. If $H^\tau$ is $k$-anisotropic, then the symmetric space  consists of semisimple elements. 
\end{thrm}

\begin{thrm}\label{Beun} 
Let $G$ be defined over $k$ and $\tau=\tau_m$ the involution of $G$. The fixed-point group of $\tau_m$, is $k$-anisotropic if and only if $m\neq 1$. 
\end{thrm}

Combing the previous two theorems, we have the following corollary. 

\begin{cor}\label{cor2}
Let $G$ be defined over a field $k$ with characteristic zero and $\tau=\tau_m$ the involution of $G$. If $m\neq 1$ then the symmetric space consists of semisimple elements. 
\end{cor}

\begin{ex}
For $G$ defined over a field $k$ with the involution $\tau_m$, the corresponding symmetric space consists of semisimple elements in the following cases:
\begin{enumerate}[label=\roman*.]
\item $k=\mathbb R$ and $m=-1$
\item $k=\mathbb F_q$ and $m=N_p$. 
\item $k=\mathbb Q_p$ and $m=p,N_p,$ or $pN_p$
\end{enumerate}
\end{ex}

\begin{remark}
To complete Theorem \ref{helmwang} and Corollary \ref{cor2} for fields of characteristic not $2$, we need to determine if the symmetric space contains only semisimple elements for the involution $\tau_m$, $m\neq 1$, over fields with prime characteristic $p$.
\end{remark}

\begin{thrm}\label{Qtildesemisimple}

Let $G$ be defined over a field $k$ and $\tau=\tau_m$ the involution of $G$. If $ m\not \in(k^*)^2$, then the extended symmetric space consists of semisimple elements. \end{thrm}

\begin{proof}
Let $q=\left(\begin{smallmatrix} a & b \\ -mb & c \end{smallmatrix}\right)\in\tildeQ$. To determine if $q$ is semisimple we analyze its eigenvalues, given in \eqref{evalues}. If $q$ has two distinct eigenvalues, then $q$ is semisimple. 
The cases of concern are when
$q$ has one eigenvalue with multiplicity $2$. 
\begin{equation}\label{evalues}\left\{\frac{1}{2}\left(a+c\pm \sqrt{c^2-2ac+a^2-4mb^2}\right)\right\}\end{equation} 
By \eqref{evalues} and $\det(q)=1$, we have a necessary and sufficient condition for $q$ to have one eigenvalue.

\begin{equation}\label{1} (a+c)^2 =4 \end{equation}

Solving \eqref{1}, $q$ has one eigenvalue if and only if $a+c=\pm 2$.  
Assume $a+c=2$, then $q=\left(\begin{smallmatrix} a & b \\ -m b & 2-a\end{smallmatrix}\right)$ and $\det(q)=1$ implies $y=\pm\frac{x-1}{\sqrt{m}}$. By assumption, $\sqrt{m}\not \in k$, which implies $a=1$ and $b=0$, i.e. $q=\Id$.
If you assume $a+c=-2$, similar calculations yield $q=-\Id$.
\end{proof}

\begin{note}
By Corollary \ref{numberofinvs}, there are two classes of involutions of $\Aut(G,G_{\mathbb F_q})$. These classes are represented by $\{\tau_1,\tau_{N_p}\}$ where $N_p$ is the smallest non-square in the field.

\end{note}

\begin{cor}\label{Qsemisimple}

Let $G$ be defined over $k$ and $\tau=\tau_m$ the involution of $G$. If $ m\not\in (k^*)^2$, then the symmetric space consists of semisimple elements. 
\end{cor}
\begin{proof}
By Lemma \ref{QinQtilde} and Theorem \ref{Qtildesemisimple}, the elements of the symmetric space must be semisimple.
\end{proof}

\begin{remark}
While combining Theorems \ref{helmwang} and \ref{Beun} proves this result for fields  with characteristic zero, our result and proof holds for any field with characteristic not $2$. 
\end{remark}

\begin{lem}\label{m1Qnotsemisimple}

Let $G$ be defined over $k$ and $\tau=\tau_1$ the involution of $G$. Then there exists elements in the symmetric space which are not semisimple. 

\end{lem}
\begin{proof}
We will construct an element in the symmetric space with a unipotent factor.
Let $g=\left(\begin{smallmatrix} x+2 & x+1 \\ -(x+1) & -x\end{smallmatrix}\right)\in G_k$ for some $x\in k \setminus \{-1\}$. Then $q=g\tau(g)^{-1} \in Q^\tau$ has a Jordan decomposition with a unipotent factor, thus $q$ is not semisimple.

\begin{equation}g\tau(g)^{-1}=\begin{pmatrix}3+2x & 2+2x \\ -(2+2x) & -(2x+1)\end{pmatrix}=S^{-1}\begin{pmatrix} 1 & 1 \\ 0 & 1 \end{pmatrix} S\end{equation}

\end{proof}

\begin{cor}
Let $G$ be defined over $k$ and $\tau=\tau_1$ the involution of $G$. Then there exists elements in the extended symmetric space which are not semisimple. 
\end{cor}

\begin{proof}
This follows from Lemmas \ref{QinQtilde} and \ref{m1Qnotsemisimple}.
\end{proof}

\begin{cor}\label{QQtildess}
Let $G$ be defined over $k$ and $\tau=\tau_m$ the involution of $G$. Then the symmetric space and extended symmetric space consist of semisimple elements if and only if $m\not\in (k^*)^2$.
\end{cor}

\begin{lem}\label{hahinQ}
Let $G$ be defined over $k$, $\tau$ an involution of $G$, and $A$ a $(\tau,k)$-split torus of $G$. Then the image of $A_k$ under conjugation by $H_k^\tau$ is contained in the extended symmetric space. 
\end{lem}
\begin{proof}

Let $a\in A_k$ and $h\in H_k^\tau$. Then $hah^{-1} \in \tildeQ$.

 \[\tau(hah^{-1})=\tau(h)\tau(a)\tau(h^{-1})=ha^{-1}h^{-1}=(hah^{-1})^{-1}\]
\end{proof}

\begin{thrm} 
Let $G$ be defined over $k$ and $\tau=\tau_m$ the involution of $G$. If $m\not \in (k^*)^2$ then the extended symmetric space decomposes as the disjoint union of the $H_k^\tau$-orbits of the maximal $k$-split tori $\{A_i \ \vert \ i \in I\}$.
\[\tildeQ=\bigcupdot_{i\in I} H_k^\tau\cdot ({{A}_i})_k\]
\end{thrm}

\begin{proof}
 Let $\tau=\tau_m$ with $m\not=1$.
 By Lemma \ref{hahinQ}, $H_k\cdot ({A_i})_k \subset \tildeQ$ for all $\{A_i\vert i\in I\}$. 
 For $q\in \tildeQ$, $q$ is $\tau$-split and semisimple by Corollary \ref{QQtildess}. Thus $q$ must be contained in the $H_k^\tau$-conjugacy class of some $k$-split torus ${(A_i)}_k$.
\end{proof}

\begin{cor} Let $G$ be defined over $k$ and $\tau=\tau_m$ the involution of $G$, If $m\not\in (k^*)^2$, then $G_k$ decomposes as 

$G_k=\displaystyle \bigcupdot_{i\in I} H_k^\tau {(A_i)}_kH_k^\tau U_k$, where $\{A_i \ \vert \ i\in I\}$ are the $H^\tau_k$-conjugacy classes of maximal $k$-split tori.

\end{cor}

\begin{notation}
Let $(Q^\tau)^{ss}$ and $(\tildeQ)^{ss}$ denote the subset of semisimple elements in the symmetric space and extended symmetric space, respectively. 
\end{notation}

\begin{lem}\label{m1semi}
Let $G$ be defined over $k$ and $\tau=\tau_1$ the involution of $G$. Then $G_k$ decomposes as $G_k=H_k^\tau(\tildeQ)^{ss}U_k$. \end{lem}

\begin{proof}
By Theorem \ref{HQR}, it suffices to show $\tildeQ \setminus \tildeQ^{ss} \in  H_k^\tau (\tildeQ)^{ss} U_k$. 
Using the construction of $q$ as in Theorem \ref{Qtildesemisimple}, let $q=\left(\begin{smallmatrix} x & x-1 \\ 1-x & 2-x \end{smallmatrix}\right)\in \tildeQ \setminus \tildeQ^{ss}$, $x\neq 1$.

 Let $h=\left(\begin{smallmatrix} a & b \\ b & a \end{smallmatrix}\right) \in H_k^\tau \setminus \{\pm\Id\}$ and $u=\left(\begin{smallmatrix} 1 & \frac{2b(-ax-bx+b)}{2b^2x-b^2+x^2} \\ 0 & 1 \end{smallmatrix}\right)\in U_k$. For $u$ to be defined, choose $a,b$ such that $x \neq -b^2\pm ba$. 
Then by \eqref{hquss}, $\tau(hqu)=(hqu)^{-1}$, hence $hqu\in\tildeQ$. 
Furthermore, the Jordan decomposition of $hqu$, $hqu=S^{-1}JS$, $J=\left(\begin{smallmatrix} f_1(a,b,x) & 0 \\ 0 & f_2(a,b,x)\end{smallmatrix}\right),$ proves $hqu$ is semisimple. 
 
\begin{equation}\label{hquss}
\tau(hgu)=\begin{pmatrix}
\frac{2 a^2 b x^2-2 a b^2 x^2-2 a^2 b x+a b^2 x-a x^3-b^3 x+b x^3+2 a x^2+b^3-b x^2}{2 b^2 x-b^2+x^2} &-ax+bx+a \\ -\frac{2a^2bx^2-2ab^2x^2+3ab^2x-ax^3-b^3x+bx^3-ab^2+ax^2-2bx^2}{2b^2x-b^2+x^2} & ax-bx+b \end{pmatrix} =(hqu)^{-1}
\end{equation}
For $k=\mathbb F_3$, one can easily verify this result holds although $H^\tau_k=\{\pm\Id\}$.

\end{proof}

We can now simplify our main result, Theorem \ref{HQR}.

\begin{cor}
Let $G$ be defined over $k$ and $\tau=\tau_m$ an involution of $G$. Then $G_k$ decomposes as $G_k=H_k^\tau(\tildeQ)^{ss}U_k$. 
\end{cor}
\begin{proof}
If $m=1$, use Lemma \ref{m1semi}. If $m\neq 1$, then $\tildeQ=(\tildeQ)^{ss}$.
\end{proof}


\section{Refining the Decomposition}\label{refining}

\subsection{Pairwise intersections of $H_k^\tau$, $\tildeQ$, and $U_k$}
\hfill

We will begin refining the decomposition by determining the pairwise intersections of the fixed-point, symmetric, and extended symmetric spaces of $G$. 
For the following propositions, let $\tau=\tau_m$ be the involution of $G$, $H^\tau$ the fixed-point group, $\tildeQ$ the extended symmetric space, and $U$ the unipotent subgroup of $G$ consisting of upper triangular matrices with $1$'s on the diagonal. 

\begin{prop}

\begin{equation}\label{hqintu} H_k^\tau\tildeQ \bigcap U_k=\left\{\begin{pmatrix} 1 & \frac{2b}{a} \\ 0 & 1 \end{pmatrix} \bigg \vert \ a\in k^*, \ b\in k, \ a^2-mb^2=1\right\}\end{equation}

\end{prop}
\begin{proof}

Let $X=\left(\begin{smallmatrix} ax-mby & ay+bz \\ m(bx-ay) & mby+az\end{smallmatrix}\right)\in H_k^\tau\tildeQ$ for some $a^2-mb^2=1$ and $xz+my^2=1$. Then $X\in U_k$ implies there exists $u\in U_k$ such that $X=u$.

\begin{equation}\label{hq=u}
\begin{pmatrix} ax-mby & ay+bz \\ m(bx-ay) & mby+az\end{pmatrix}=\begin{pmatrix} 1 & u_1 \\ 0 & 1\end{pmatrix}
\end{equation}
 Solving \eqref{hq=u}, we obtain $u_1=\frac{2b}{a}$.
\end{proof}

\begin{remark} In general, $H_k^\tau \tildeQ \bigcap U_k$ is not contained in $G_k$. The order of \eqref{hqintu} is equivalent to the order of $H_k^\tau$ minus the elements of $H_k^\tau$ with zeroes on the diagonal.  
\[\vert H_k^\tau \tildeQ \bigcap U_k \vert =\vert H_k^\tau \vert-\left\vert \left\{ b \in k^* \  \vert \ b=\pm \frac{1}{\sqrt{-m}}\right\}\right\vert\]

\end{remark}

\begin{prop}

\[H_k^\tau\bigcap \tildeQ= U_k\bigcap H_k^\tau= U_k\bigcap \tildeQ=\pm\Id\]
\end{prop}
\begin{proof}
This is clear by the definitions of $H_k^\tau, \tildeQ$, and $U_k$.

\end{proof}

\begin{prop}\label{HintQR}

\begin{equation}\label{hintqr}H_k^\tau \bigcap \tildeQ U_k=\left\{ \begin{pmatrix} a & b \\ mb & a\end{pmatrix} \bigg \vert  \ a\in k^*, \ b\in k, \ a^2-mb^2=1\right\}\end{equation}
\end{prop}

\begin{proof}
Let $X=\left(\begin{smallmatrix} x & u_1x + y \\ -my & -myu_1+z\end{smallmatrix}\right)\in \tildeQ U_k$, for some $xz+my^2=1$.
Then $X\in H_k^\tau$ implies there exists $h\in H_k^\tau$ such that $X=h$.

\begin{equation}\label{qu=h}\begin{pmatrix} x & u_1 x + y \\ -my & -myu_1+z\end{pmatrix}=\begin{pmatrix} a & b \\ mb & a \end{pmatrix}\end{equation}

Solving \eqref{qu=h} we obtain $x=a$, $y=-b$ and $u_1=\frac{2b}{a}$.

\end{proof}
\begin{remark}

The size of \eqref{hintqr} is the order of $H_k^\tau$ minus the elements in $H_k^\tau$ with zeroes on the diagonal. Furthermore, the size of \eqref{hintqr} is equivalent to the size of \eqref{hqintu}.

 \[\vert(H_k^\tau\bigcap \tildeQ U_k) \vert= \vert H_k^\tau \tildeQ \bigcap U_k \vert =\left\vert H_k^\tau\right\vert-\left\vert\left\{b\in k^* \ \vert \ b=\pm \frac{1}{\sqrt{-m}}\right\}\right\vert\] 

This intersection is almost equivalent to $H_k^\tau$.
\[H_k^\tau\setminus(H_k^\tau\bigcap \tildeQ U_k)=\left\{\begin{pmatrix} 0 & b \\ mb & 0 \end{pmatrix} \bigg \vert \ b\in k, \ -mb^2=1\right\}\]

\end{remark}
 
 \begin{lem}\label{HinQU}
 Let $G$ be defined over $k$ and $\tau=\tau_m$ the involution of $G$. The fixed-point group of $\tau_m$ is contained in $\tildeQ U_k$ if and only if $-m \not \in (k^*)^2$.
 \end{lem}

\begin{proof}
This proof follows from the following chain of equivalent statements.
\begin{align*}
H_k^\tau & \subset \tildeQ U_k \\
H_k^\tau\setminus(H_k^\tau\bigcap \tildeQ U_k)&=\left\{\left(\begin{smallmatrix} 0 & b \\ mb & 0 \end{smallmatrix}\right)  \big\vert \ b\in k,\ -mb^2=1\right\} =\emptyset\\
b&=\pm \frac{1}{\sqrt{-m}}  \not\in k
\end{align*}
\end{proof}

\begin{ex}When $H_k^\tau\subset \tildeQ U_k$, we do not necessarily have $G_k=\tildeQ U_k$. Let $G$ be defined over $k=\mathbb R$ and $\tau=\tau_1$ the involution of $G$.
By Lemma \ref{HinQU}, $H_k\subset \tildeQ U_k$. Let $g=\left(\begin{smallmatrix} 0 & \frac{1}{2} \\ -2 & 0 \end{smallmatrix}\right)\in G_k$ then $gu\not\in \tildeQ$
for any $u=\left(\begin{smallmatrix} 1 & u_1 \\ 0 & 1 \end{smallmatrix}\right)\in U_k.$
 Therefore $G_k\not= \tildeQ U_k$.
\[\tau(gu)=\begin{pmatrix} -2u_1 & -2 \\ \frac{1}{2} & 0 \end{pmatrix} \not =\begin{pmatrix} -2u_1 & -\frac{1}{2} \\ 2 & 0 \end{pmatrix}=(gu)^{-1}\]
\end{ex}

\begin{prop}\label{HRintQ}

\begin{equation}\label{hrintq}H_k^\tau U_k\bigcap \tildeQ=\left\{\begin{pmatrix} x & y \\ -my & z \end{pmatrix} \bigg\vert \ x\in k^*, \ y,z\in k, \ xz+my^2=1\right\}\end{equation}
\end{prop}

\begin{proof}
Let $X=\left(\begin{smallmatrix} a & u_1 a + b \\ mb & mbu_1+a\end{smallmatrix}\right)\in H_k^\tau U_k$, for some $a^2-mb^2=1$. 
Then $X\in \tildeQ$ implies there exists $q\in \tildeQ$ such that $q=X$.

\begin{equation}\label{hu=q}
\begin{pmatrix} a & u_1 a + b \\ mb & mbu_1+a \end{pmatrix}=\begin{pmatrix} x & y \\-my & z \end{pmatrix}
\end{equation}

Solving \eqref{hu=q}, we obtain $a=x$, $b=-y,$ and $u_1=\frac{2y}{x}$.

\end{proof}
\begin{remark}\label{huintqex}
Similar to the previous example, \eqref{hrintq} is almost equivalent to the extended symmetric space. 

\[\tildeQ \setminus (H_k^\tau U_k \bigcap \tildeQ)=\left\{\begin{pmatrix} 0 & y \\ -my & z \end{pmatrix} \bigg \vert \ y\in k, \ my^2=1\right\}\]

\end{remark}

\begin{lem}\label{QinHR}

Let $G$ be defined over $k$ and $\tau=\tau_m$ the involution of $G$. Then the extended symmetric space is contained in $H_k^\tau U_k$ if and only if $m\not\in (k^*)^2$.

\end{lem}

\begin{proof}
This proof follows from the following chain of equivalent statements.
\begin{align*}
\tildeQ & \subset H_k^\tau U_k \\
\tildeQ \setminus (H_k^\tau U_k \bigcap \tildeQ)&=\left\{\begin{pmatrix} 0 & y \\ -my & z \end{pmatrix} \bigg \vert \ y\in k, \ my^2=1\right\} =\emptyset\\
y&=\pm \frac{1}{\sqrt{m}}  \not\in k
\end{align*}
\end{proof}

\subsection{Generalization of the Iwasawa Decomposition}
\hfill

\begin{thrm}\label{HR}
Let $G$ be defined over $k$ and $\tau =\tau_m$ the involution of $G$. If $m\not \in (k^*)^2$, then $G_k=H_k^\tau U_k$.
\end{thrm}

\begin{proof}
Let $g\in G_k$ and $\tau=\tau_m$ the involution of $G$ with $m\not\in (k^*)^2$. By Theorem \ref{HQR} write $g=hqu$ for some $h\in H_k^\tau, q\in \tildeQ$ and $u\in U_k$. 
By Lemma \ref{QinHR}, write $q=h_1u_1$ for some $h_1\in H_k^\tau$ and $u_1\in U_k$. 
Thus, $g=hh_1u_1u\in H_k^\tau U_k$. 
The reverse containment in clear. 
\end{proof}

\begin{thrm}\label{HwR}

Let $G$ be defined over $k$ and $\tau=\tau_1$ the involution of $G$. Then $\displaystyle G_k=\bigcup_{\omega\in W(T)} H_k^\tau\omega U_k$, where $W(T)$ is the Weyl group of the maximal $k$-split torus. 
\end{thrm}

\begin{proof}
Let $W(T)$ be as in \eqref{bruhat} and $g\in G_k$. 
If $g\in H_k^\tau U_k \bigcap \tildeQ$, then $g\in H_k^\tau U_k$. If $g\not\in H_k^\tau U_k \bigcap \tildeQ$, write $g=hqu$ as in Theorem \ref{HQR}, where $q \in \tildeQ \setminus (H_k^\tau U_k \bigcap \tildeQ)$.
By Remark \ref{huintqex}, let $q=\left(\begin{smallmatrix} 0 & 1 \\ -1 & z \end{smallmatrix}\right)$ without loss of generality and $u_1=\left(\begin{smallmatrix} 1 & -z \\ 0 & 1 \end{smallmatrix}\right)$, then $g\in H_k^\tau \omega U_k$, $\omega\in W(T)$.

\[g=hqu=h\begin{pmatrix} 0 & 1 \\ -1 & 0 \end{pmatrix} u_1 u \in H_k^\tau \begin{pmatrix} 0 & 1 \\ -1 & 0 \end{pmatrix}U_k\]

\end{proof}

\subsection{Commutativity of the factors of Theorem \ref{HQR} }
\hfill

For the involution $\tau$ of $G$ in the following lemmas, let $H^\tau$, $\tildeQ$, and $U$ be the fixed-point group, the extended symmetric space, and the unipotent subgroup, respectively.
\begin{lem}\label{HQ=QH}
Let $G$ be defined over $k$ and $\tau$ an involution of $G$. Then $H_k^\tau\tildeQ=\tildeQ H_k^\tau$.

\end{lem}

\begin{proof}
Let $g\in \tildeQ H_k^\tau$, then $g=q_1h_1$ for some $q_1\in\tildeQ$ and $h_1\in H_k^\tau$. 
Using Lemma \ref{hahinQ}, $g\in H_k^\tau \tildeQ$.
\[ g=q_1h_1=h_1(h_1^{-1}q_1h_1) \in H_k^\tau\tildeQ\]

Similarly, let $g=h_1q_1\in H_k^\tau\tildeQ$, then $g\in \tildeQ H_k^\tau$.
\[
g=h_1q_1=(h_1q_1h_1^{-1})h_1 \in \tildeQ H_k^\tau
\]

\end{proof}

\begin{remark} In general, $U_kH_k^\tau \not = H_k^\tau U_k$ and $U_k\tildeQ\not= \tildeQ U_k$. \end{remark}

\begin{lem}\label{UHQ}
Let $G$ be defined over $k$ and $\tau$ an involution of $G$, 
then $G_k=U_kH_k^\tau\tildeQ$.
and $G_k=\tildeQ U_kH_k^\tau$.
\end{lem}

The proof of Lemmas \ref{UHQ} follows the same technique as the proof of Theorem \ref{HQR}, using the Bruhat decomposition. It is left to the reader to show that for all $g\in G_k$, there exists $u_1,u_2\in U_k$ and $h_1,h_2\in H_k^\tau$ such that $h_1u_1g \in \tildeQ$ and $gh_2u_2 \in \tildeQ$.

\begin{cor}

Let $G$ be defined over $k$ and $\tau$ an involution of $G$. The following decompositions of $G_k$ are equivalent:
\begin{enumerate}[label=\roman*.]
\begin{multicols}{3}
\item $G_k=H_k^\tau\tildeQ U_k$
\item $G_k=H_k^\tau U_k \tildeQ$
\item $G_k=\tildeQ H_k^\tau U_k$
\item $G_k=\tildeQ U_k H_k^\tau$
\item $G_k=U_kH_k^\tau \tildeQ$
\item $G_k=U_k\tildeQ H_k^\tau$
\end{multicols}
\end{enumerate}

\end{cor}

This corollary combines Lemmas \ref{HQ=QH}, \ref{UHQ}, and Theorem \ref{HQR}.


\section{Summary of Results over Fields of Characteristic $2$}\label{char2}
\begin{notation}
In this section,
let $k$ be a field of characteristic $2$. From \cite{nathaniel}, we have results concerning the isomorphy classes of involutions of $\Aut(G,G_k)$ for a field $k$ with characteristic $2$. 
\end{notation}
\begin{note}
In a field $k$ with characteristic $2$, we have $x=-x$ for all $x$ in $k$.
\end{note}

\begin{thrm}
If $k$ is a finite field or algebraically closed, then there is one isomorphism class of involutions of $\Aut(G,G_k)$. 
\end{thrm}

\begin{notation} For $k=\mathbb F_2^r$ or $k=\bar k$, we will represent this isomorphy class of involutions by $\tau_0=\Inn(N),$ where $N=\left(\begin{smallmatrix}1 & 1 \\ 0 & 1 \end{smallmatrix}\right)$.
\end{notation}

\begin{remark} The fixed-point group of $\tau_0$ is the unipotent subgroup $H_k^{\tau_0}=\left\{\left(\begin{smallmatrix} a & b \\ 0 & a \end{smallmatrix}\right) \big \vert a^2=1 \right\}$. The extended symmetric space of $\tau_0$ is $\widetilde{Q^{\tau_0}}=\left\{\left(\begin{smallmatrix} x & y \\ z & x+z \end{smallmatrix}\right) \big \vert x^2+xz+yz=1\right\}$.
\end{remark}

Because the fixed-point group of $\tau_N$ is unipotent, we will not need to include the unipotent subgroup $U$ from Theorem \ref{HQR}. 
\begin{thrm}
Let $G$ be defined over an algebraically closed field or finite field $k$ and $\tau=\tau_0$ the involution of $G$. For the fixed-point group $H^\tau_k$, extended symmetric space $\tildeQ$, and Weyl group of the maximal $k$-split torus $W(T)$, we can factor the group as $G_k=\displaystyle \bigcup_{\omega\in W(T)} H_k^\tau \omega \tildeQ$.
\end{thrm}

\begin{proof}
Let $W(T)$ be as in \eqref{bruhat}. For $g=\left(\begin{smallmatrix} a & b \\ c & d \end{smallmatrix}\right)\in G_k$ with $c\in k^*$, 
let $h=\left(\begin{smallmatrix} 1 & \frac{a+c+d}{c} \\ 0 & 1 \end{smallmatrix}\right)\in H_k^\tau$ and $q=\left(\begin{smallmatrix} c+d & \frac{1+cd+d^2}{c} \\ c & d   \end{smallmatrix}\right) \in \tildeQ$, then $g=hq\in H_k^\tau\tildeQ$.
For $g=\left(\begin{smallmatrix} a & b \\ 0 & \frac{1}{a} \end{smallmatrix}\right)\in G_k$, let $h=\left(\begin{smallmatrix} 1 & a^2 + ab \\ 0 & 1 \end{smallmatrix}\right) \in H_k^\tau$ and $q=\left(\begin{smallmatrix} 0 & \frac{1}{a} \\ a & a \end{smallmatrix}\right) \in \tildeQ$, then $g=h\omega q\in H_k^\tau \omega \tildeQ$, where $\omega=\left(\begin{smallmatrix} 0 & 1 \\ 1 & 0 \end{smallmatrix}\right)$.  
\end{proof}

\begin{thrm}
If $k$ is an infinite field which is not algebraically closed, then there is an infinite number of isomorphy classes of involutions of $\Aut(G,G_k)$. 

\end{thrm}
\begin{notation} For an infinite field $k$ such that $k\neq \bar k$, we will represent the isomorphy classes of involutions by $\tau_m=\Inn(M),$ where $M=\left(\begin{smallmatrix}0 & 1 \\ m & 0 \end{smallmatrix}\right)$.
\end{notation}

\begin{remark}
The fixed-point group and extended symmetric space of $G$ defined over $k$ with the involution $\tau_m$ are as follows:
\[H_k^\tau =\left\{\begin{pmatrix} a & b \\ mb & a \end{pmatrix} \bigg \vert \ a,b\in k, \ a^2+mb^2=1\right\} \quad \tildeQ=\left\{\begin{pmatrix}a& b \\ mb & c\end{pmatrix}\bigg \vert \ a,b,c\in k, \ ac+mb^2=1\right\}\]
\end{remark}

\begin{thrm}
Let $G$ be defined over a field $k$ with characteristic $2$ and $\tau=\tau_m$ the involution of $G$. For the fixed-point group $H^\tau$, extended symmetric space $\tildeQ$, and the unipotent subgroup $U_k=\left\{ \left(\begin{smallmatrix} 1 & u_1 \\ 0 & 1 \end{smallmatrix}\right) \big \vert \ u \in k \right\}$, $G_k=H_k^\tau \tildeQ U_k$.
\end{thrm}

\begin{proof}

Let $g=\left(\begin{smallmatrix} x& y \\ z & w \end{smallmatrix}\right)\in G_k$. 
If $z\neq 0$, for $u=\left(\begin{smallmatrix} 1 & \frac{x+w}{z} \\ 0 & 1 \end{smallmatrix}\right)\in U_k$ and $h=\left(\begin{smallmatrix} 0 & b \\ m b & 0 \end{smallmatrix}\right)\in H^\tau_k$, we have $\tau(hgu)=(hgu)^{-1}$. Hence, $hgu\in\tildeQ$ and $g\in H_k^\tau \tildeQ U_k$.

If $z\neq 0$, for $u=\left(\begin{smallmatrix} 1 & \frac{y}{x} \\ 0 & 1\end{smallmatrix}\right)\in U_k$, we have $\tau(gu)=(gu)^{-1}$. Hence, $gu\in \tildeQ$ and $g\in H_k^\tau \tildeQ U_k$. 
The reverse containment is clear. 
\end{proof}

\bibliography{Mandi}
\bibliographystyle{amsalpha}

\end{document}